\documentclass[pdflatex,sn-mathphys-num]{sn-jnl}
\usepackage{enumerate}
\usepackage{graphicx}%
\usepackage{multirow}%
\usepackage{amsmath,amssymb,amsfonts}%
\usepackage{amsthm}%
\usepackage{mathrsfs}%
\usepackage[title]{appendix}%
\usepackage{xcolor}%
\usepackage{textcomp}%
\usepackage{manyfoot}%
\usepackage{booktabs}%
\usepackage{algorithm}%
\usepackage{algorithmicx}%
\usepackage{algpseudocode}%
\usepackage{listings}%
\usepackage{stmaryrd}
\usepackage{array}
\usepackage{float}
\usepackage{placeins}

\theoremstyle{thmstyleone}%
\newtheorem{theorem}{Theorem}
\newtheorem{proposition}[theorem]{Proposition}%
  \newtheorem{corollary}{Corollary}
  \newtheorem{lemma}{Lemma}
\theoremstyle{thmstyletwo}%
\newtheorem{example}{Example}%
\newtheorem{remark}{Remark}%

\theoremstyle{thmstylethree}%
\newtheorem{definition}{Definition}%

\newcommand{\m}{\mathfrak{m}}
\newcommand{\R}{\mathbb{R}}
\newcommand{\C}{\mathbb{C}}
\newcommand{\Q}{\mathbb{Q}}
\newcommand{\Z}{\mathbb{Z}}
\newcommand{\N}{\mathbb{N}}
\newcommand{\Na}{\mathbb{N}_0[\alpha]}

\begin{document}

\title[Finite Generation in Polynomial Semirings]{Finite Generation in Polynomial Semirings}

\author[1]{\fnm{Mohammad} \sur{El-Asal}}\email{12531180@students.liu.edu.lb}

\author[2]{\fnm{Wael} \sur{Mahboub}}\email{wael.mahboub@lau.edu.lb}

\affil[1]{\orgdiv{Physics and Mathematics}, \orgname{Lebanese International University}, \orgaddress{\city{Beirut}, \country{Lebanon}}}

\affil[2]{\orgdiv{Mathematics}, \orgname{Lebanese University}, \orgaddress{\street{Fanar}, \city{Beirut}, \country{Lebanon}}}

\abstract{We study the semiring
$\mathbb{N}_0[\alpha]$ as an additive monoid where $\alpha$ is a positive real algebraic number. In the atomic case, the atoms of
$\mathbb{N}_0[\alpha]$ are precisely the powers $\alpha^n$ up to a  certain nonnegative integer $n$, and finite generation is governed
by divisibility of the minimal polynomial by a negative-tail polynomial. Our first main result gives a complete
characterization when the minimal polynomial has the form $\mathfrak{m}_\alpha(X)=p_\alpha(X)-c$ with
$c\in\mathbb{N}$.  Our second main result shows that finite generation forces $\alpha$ to be a weak Perron number, and proves a converse under the additional assumptions that $\alpha$ is an algebraic integer and the unique positive conjugate of its minimal polynomial. As an application, we analyze cubic minimal
polynomials and obtain a partial classification of rank-$3$ monoids $\mathbb{N}_0[\alpha]$ by generation and
factorization type, including coefficient constraints, non--length-factoriality results for a large family, and
examples with prescribed numbers of atoms.}


\maketitle

\section{Introduction}\label{sec1}

The study of \emph{semirings} is usually traced to
Vandiver's 1934 paper, where the algebraic consequences of dropping additive cancellation (and,
more generally, additive inverses) were investigated \cite{Vandiver1934}.
Their ideal theory was developed early on by Bourne, who introduced and studied analogues of radicals
in semirings \cite{Bourne1951}. By now, semiring theory has matured into a broad subject with
several standard references and applications, including the monographs of Golan and of
Hebisch-Weinert \cite{Golan1999,HebischWeinert1998}.

We study the polynomial semiring $\mathbb{N}_0[X]$. Fixing $\alpha\in\mathbb{R}_{>0}$, evaluation at
$\alpha$ defines an additive submonoid
\[
\mathbb{N}_0[\alpha]=\{f(\alpha): f(X)\in \mathbb{N}_0[X]\}\subseteq \mathbb{R}_{\ge 0}.
\]
Recent work of Correa-Morris and Gotti shows that such monoids arise naturally in the study of \emph{algebraic valuations of polynomial semirings} and develops a detailed understanding of their atomic structure and factorization phenomena \cite{CMG}. Recent progress can be found in \cite{ajran2023factorization,chapman2026betti,jiang2023primality,gotti2023subatomicity}. This places $\mathbb{N}_0[\alpha]$ within the general framework of factorization theory in commutative cancellative monoids, a subject whose modern development grew out of the failure of unique factorization in algebraic number theory and is now treated in the monograph of Geroldinger and Halter-Koch. \cite{GHK2006}.

In \cite{dani2025set}, the authors study the admissibility of pairs $(\left|\mathcal{A}(\Na)\right|,\left|\mathcal{S}(\Na)\right|)$ where the latter denotes the number of strong atoms.
In the same paper, the authors give a detailed classification of rank $2$ semirings of the form $\Na$. We note that Theorem 5.4 of \cite{dani2025set} yields a partial answer to the conjecture stated in \cite{CMG} (Conjecture 5.14): For every $n\in\N$ there is an algebraic number $\alpha\in\R_{>0}$ such that $\left|\mathcal{A}(\Na)\right|-
\deg \m_\alpha(X) = n$, for even integers $n$.

In this paper, we aim to study the cardinality of the set of atoms $\left|\mathcal{A}(\Na)\right|$. Our first main result is a complete characterization of polynomials of the form $\m_\alpha(X)=p(X)-c$, where $p(X)\in \N_0[X]$ (see Theorem~\ref{pminusc}). Our next main result links the weak Perron numbers to infinitely generated monoids; this is Theorem~\ref{perron}.

Next, we apply our results to obtain a partial classification of rank-$3$ monoids of the form $\Na$ according to their generation pattern; this is done in Section~\ref{deg3}. Most cases follow directly from the theorems in Section~\ref{general}. For the form $\m_\alpha(X) = X^3 + aX^2 - bX - c$, we give necessary conditions for finite generation of $\Na$ and provide a family of monoids showing that these conditions are not sufficient. We also prove that it is never an $LFM$, and we give a family of monoids with $\left|\mathcal{A}(\Na)\right| = 5$.
For the form $\m_\alpha(X) = X^3 - aX^2 + bX - c$, we provide a necessary and sufficient condition for infinite generation and necessary and sufficient conditions for $\Na$ to be a proper $LFM$. 

\section{Notation}

\begin{itemize}
\item We write $\N_0:=\N\cup\{0\}$.
\item For integers $m\le n$ we use $\llbracket m,n\rrbracket := \{m,m+1,\dots,n\}$.
\item If $(M,+,0)$ is a commutative monoid and $S\subseteq M$, then
\[
\langle S\rangle \ :=\ \left\{\sum_{i=1}^k s_i \ :\ k\in\N,\ s_i\in S\right\}\cup\{0\}
\]
denotes the submonoid generated by $S$.
\end{itemize}

\begin{definition}
A \emph{commutative semiring} is a set $R$ with operations $(+, \cdot)$ such that
$(R,+,0)$ is a commutative monoid, $(R,\cdot,1)$ is a commutative monoid,
multiplication distributes over addition, and $0\cdot r=r\cdot 0=0$ for all $r\in R$.\\

The semiring $\N_0[X]$ is the set of polynomials with coefficients in $\N_0$ with the usual
addition and multiplication of polynomials.
\end{definition}

\begin{definition}
Fix $\alpha\in\R_{>0}$. The evaluation map
\[
\operatorname{ev}_\alpha:\N_0[X]\to \R_{\ge 0},\qquad f(X)\mapsto f(\alpha),
\]
has image
\[
\N_0[\alpha]\ :=\ \operatorname{ev}_\alpha(\N_0[X]) \ \subseteq\ (\R_{\ge0},+),
\]
which we view as a commutative cancellative monoid under addition.
\end{definition}

\begin{definition}
Let $m_\alpha(X)\in\Q[X]$ be the \emph{minimal polynomial} of $\alpha$ over $\Q$.
Let $\m_\alpha(X)\in\Z[X]$ be the unique \emph{primitive} integer polynomial with positive leading
coefficient whose image in $\Q[X]$ is $m_\alpha(X)$ (so $\m_\alpha(\alpha)=0$ and $\gcd$ of its coefficients is $1$).

For any $f(X)=\sum_i a_iX^i\in\Z[X]$, define its positive and negative parts
\[
f^+(X):=\sum_{a_i>0} a_iX^i\in\N_0[X],\qquad
f^-(X):=\sum_{a_i<0} (-a_i)X^i\in\N_0[X],
\]
so that $f=f^+-f^-$ and $\operatorname{supp}(f^+)\cap\operatorname{supp}(f^-)=\varnothing$.
We call $(f^+,f^-)$ the \emph{(canonical) minimal pair} of $f$.
In particular, we write $\m_\alpha = p_\alpha - q_\alpha$ with $(p_\alpha,q_\alpha)=(\m_\alpha^+,\m_\alpha^-)$.
\end{definition}

\begin{definition}
Let $(M,+,0)$ be a commutative cancellative monoid.
An element $u\in M$ is a \emph{unit} if it has an additive inverse in $M$.
A nonunit $a\in M\setminus\{0\}$ is an \emph{atom} (or \emph{irreducible}) if
$a=b+c$ with $b,c\in M$ implies that $b$ or $c$ is a unit.
We write $\mathcal A(M)$ for the set of atoms of $M$.
The monoid $M$ is \emph{atomic} if every nonzero nonunit is a finite sum of atoms,
and \emph{antimatter} if $\mathcal A(M)=\varnothing$.
\end{definition}

\begin{definition}
Assume $M$ is atomic. A \emph{factorization} of $x\in M\setminus\{0\}$ is an expression
$x=a_1+\cdots+a_k$ with $k\in\N$ and $a_i\in\mathcal A(M)$.
Two factorizations are identified if they differ only by permuting the atoms.
Let $\mathsf Z(x)$ be the set of factorizations of $x$, and for $z\in\mathsf Z(x)$ let $|z|:=k$
be its \emph{length}. The \emph{set of lengths} is $\mathsf L(x):=\{|z|:z\in\mathsf Z(x)\}$.

\begin{itemize}
\item $M$ is a \emph{UFM} (unique factorization monoid) if $|\mathsf Z(x)|=1$ for every $x\ne 0$.
\item $M$ is an \emph{HFM} (half-factorial monoid) if $|\mathsf L(x)|=1$ for every $x\ne 0$.
\item $M$ is an \emph{LFM} (length-factorial monoid) if the map $\mathsf Z(x)\to\N$, $z\mapsto |z|$ is injective
for every $x\ne 0$.
We say $M$ is a \emph{proper LFM} if it is an LFM but not a UFM.
\item $M$ is an \emph{FFM} (finite factorization monoid) if $\mathsf Z(x)$ is finite for every $x\ne 0$.
\end{itemize}
\end{definition}

\begin{definition}
A commutative monoid $M$ is \emph{finitely generated} (FGM) if $M=\langle S\rangle$ for some finite $S\subseteq M$.
\end{definition}

\begin{definition} \cite{Lind1984}
Let $\alpha$ be a real algebraic number and let $\mathrm{Conj}(\alpha)$ be the set of Galois conjugates of $\alpha$ in $\C$.
We say $\alpha$ is a \emph{weak Perron number} if $\alpha>0$ and $|\beta|\le \alpha$ for every $\beta\in\mathrm{Conj}(\alpha)$. We say $\alpha$ is a \emph{Perron number} if $\alpha>0$ and $|\beta|< \alpha$ for every $\beta\in\mathrm{Conj}(\alpha)$ such that $\beta \neq \alpha$.

\end{definition}

\begin{definition}
A monic polynomial $H(X)\in \mathbb Z[X]$ is called a \emph{negative-tail polynomial} if
\[
H(X)=X^n-\sum_{i=0}^{n-1}a_iX^i
\]
for some $n\in\mathbb N$ and coefficients $a_0,\ldots,a_{n-1}\in\mathbb N_0$.
\end{definition}

Throughout this paper, we consider only commutative cancellative monoids. 

\section{Preliminaries}

Let $\alpha$ be a positive real algebraic number over $\Q$. Let $\m_{\alpha}$ be the primitive polynomial of minimal degree over $\Z$ with $\m_{\alpha}(\alpha)=0$. Let $p_{\alpha}(X),\ q_\alpha(X)\in\N_0[X]$ such that $\m_{\alpha}(X)=p_{\alpha}(X)-q_{\alpha}(X)$.\\

\begin{proposition}[\cite{CMG}, Theorem 4.2]\label{generation}
For each algebraic $\alpha \in \mathbb{R}_{>0}$, the monoid $\mathbb{N}_0[\alpha]$ is atomic if and only if $1 \in \mathcal{A}(\mathbb{N}_0[\alpha])$, and $\mathbb{N}_0[\alpha]$ is antimatter otherwise. 

Also, if $\mathbb{N}_0[\alpha]$ is atomic, then there exists $\sigma \in \mathbb{N} \cup \{\infty\}$ such that
\begin{equation}\label{eq:atoms}
\mathcal{A}(\mathbb{N}_0[\alpha]) = \{ \alpha^n : n \in [0,\sigma) \cap \mathbb{N}_0 \}.
\end{equation}

If $\mathbb{N}_0[\alpha]$ is finitely generated (and so atomic), then
\[
\sigma = \min \left\{ n \in \mathbb{N} : \alpha^n \in \langle \alpha^j : j \in \llbracket 0, n-1 \rrbracket \rangle \right\}.
\]
\end{proposition}

We deduce that, whenever $\Na$ is an FGM, we have $\sigma\geq \deg \m_\alpha$.

Note that $\Na$ is finitely generated if and only if $\Na$ is atomic and $\m_\alpha$ divides a negative tail polynomial.\\

By \cite{CMG} Theorem 5.4, we have $\sigma=\deg \m_\alpha$ if and only if $\Na$ is a UFM if and only if $\m_\alpha$ is a negative tail polynomial. We also recall the following results that we will directly use in this paper:

\begin{proposition}[\cite{CMG}, Proposition 4.5]\label{atomicity}
Let $\alpha \in \mathbb{R}_{>0}$ be an algebraic number with minimal polynomial $\m_{\alpha}(X)$. Then the following statements hold.
\begin{enumerate}
    \item If $\alpha \notin \mathbb{Q}$ and $|\m_{\alpha}(0)| \neq 1$, then $\mathbb{N}_0[\alpha]$ is atomic.
    \item If $\m_{\alpha}(X)$ has more than one positive root, then $\mathbb{N}_0[\alpha]$ is atomic.
\end{enumerate}
\end{proposition}

\begin{proposition}[\cite{CMG}, Proposition 5.6]\label{onepositiveroot}
If $\mathbb{N}_0[\alpha]$ is a finitely generated monoid (FGM) for some algebraic $\alpha \in \mathbb{R}_{>0}$, 
then $m_\alpha(X) \in \mathbb{Z}[X]$ and its only positive root is $\alpha$ (counting multiplicity).
\end{proposition}

\begin{proposition}[\cite{CMG}, Theorem 5.4]\label{ufm}
Let $\alpha \in \mathbb{R}_{>0}$ be an algebraic number. The following statements hold:
\begin{enumerate}[(1)]
    \item If $\mathbb{N}_0[\alpha]$ is a UFM, then it is finitely generated.
    
    \item Suppose that $\alpha$ has algebraic degree $d$, minimal polynomial $m_{\alpha}(X)$, and minimal pair $(p(X), q(X))$. Then the following conditions are equivalent:
    \begin{enumerate}[(a)]
        \item $\mathbb{N}_0[\alpha]$ is a UFM;
        \item $\mathbb{N}_0[\alpha]$ is an HFM;
        \item $\deg m_{\alpha}(X) = |\mathcal{A}(\mathbb{N}_0[\alpha])|$;
        \item $p(X) = X^{d}$ for some $d \in \mathbb{N}$.
    \end{enumerate}
\end{enumerate}
\end{proposition}

\begin{theorem}[\cite{CMG}, Theorem 5.9]\label{lfm}
Let $\alpha \in \mathbb{R}_{>0}$ be an algebraic number. The following conditions are equivalent:
\begin{enumerate}[(a)]
    \item $\mathbb{N}_0[\alpha]$ is a proper LFM;
    \item $\mathcal{A}(\mathbb{N}_0[\alpha]) = \{\alpha^{j} : j \in \llbracket 0, \deg m_{\alpha}(X) \rrbracket \}$.
\end{enumerate}
\end{theorem}

We also recall the following theorem:

\begin{theorem}[Descartes' Rule of Signs]\label{descarte}
Let
\[
f(X) = a_n X^n + \cdots + a_1 X + a_0 \in \mathbb{R}[X],
\]
with $a_n \neq 0$. The number of positive roots of $f$, counted with multiplicity, is at most the number of sign variations in the coefficient sequence $(a_n, \ldots, a_0)$, and differs from it by an even integer.
\end{theorem}

We can directly deduce from this rule the following lemma:

\begin{lemma}
Suppose that $\m_\alpha(X)$ has exactly two sign variations. If $\m_\alpha$ has a positive root and $\Na$ is atomic, then $\Na$ is infinitely generated.
\end{lemma}

\begin{proof}
By the theorem above, the number of positive roots of $\m_\alpha$, counted with multiplicity, is even. Since $\m_{\alpha}$ has a positive root, it must have at least two positive roots, counted with multiplicity.\\

Because $\m_\alpha$ is irreducible over $\mathbb{Q}$ and the ground field has characteristic $0$, the polynomial is separable. Thus all roots are simple, so it has at least two distinct positive roots, then by proposition \ref{atomicity} and proposition \ref{onepositiveroot} $\Na$ is infinitely generated.
\end{proof}

\begin{theorem}[\cite{Ferguson1997}]\label{ferguson}
Suppose that the irreducible polynomial $f(x)\in \mathbb{Z}[x]$ has $m$ roots, at least one of which is real, on the circle $|z|=c$ Then $f(x)=g(x^{m})$, where $g(x)$ has no more than one real root on any circle in $\mathbb{C}$.
\end{theorem}

\begin{proposition}[\cite{ChenGottiLuYao2026}, Proposition 4.4]\label{simple-polynomial}
Let $\alpha \in \mathbb{A}\cap(0,1)$. If $\alpha^{-1}$ is a Perron number with no positive conjugate aside from itself, then there exists a polynomial $h(x)\in \mathbb{Z}[x]$ such that $h(x)m_\alpha(x)\in x\mathbb{N}_0[x]-1$ and $h(x)m_\alpha(x)$ is simple.
\end{proposition}

\begin{proposition}[\cite{dani2025set}, Proposition 5.1]\label{x_k_atoms}
Let $\alpha$ be an algebraic number with minimal polynomial $m(x)\in \mathbb{Q}[x]$ such that $m(x^k)$ is irreducible in $\mathbb{Q}[x]$ for some $k\in \mathbb{N}_{\ge 2}$. If $\beta$ is a root of $m(x^k)$, then $\left|\mathcal{A}(\N_0[\beta])\right| = k\left|\mathcal{A}(\Na)\right|$.
\end{proposition}
\section{General results}\label{general}

\begin{theorem}\label{pminusc}
    Write $\m_\alpha(X)=p_\alpha(X)-q_\alpha(X)$ with $q_\alpha(X)=c\in\N$. 

    \begin{enumerate}
        \item If $p_\alpha(X)=X^m$, then $\Na$ is a UFM.
        \item Otherwise, we have two cases:
        \begin{itemize}
            \item[i)] If $c=1$ then $\Na$ is antimatter.
            \item[ii)] If $c>1$ then $\Na$ is atomic and $\Na$ is infinitely generated. 
        \end{itemize} 
    \end{enumerate}
\end{theorem}
\begin{proof}
\begin{enumerate}
    \item by Proposition~\ref{ufm}.

    \item \begin{itemize}
            \item[i)] If $c=1$, we deduce from $\m_\alpha(\alpha)=0$ that $1=p_\alpha(\alpha)$, thus by Proposition~\ref{generation}, $\Na$ is antimatter. 
            
            \item[ii)]If $c>1$, by Proposition~\ref{atomicity}, $\Na$ is atomic.
    
    Let $p_\alpha(X)=X^m+p_{m-1}X^{m-1}+\dots+p_1X$ where $(p_1,\dots,p_{m-1})$ is a non-zero vector in $\N_0^{m-1}$. Suppose by contradiction that $\Na$ is finitely generated, then there exists a negative tail polynomial $H(X)$ such that $H(X)=\m_\alpha(X) \cdot Q(X)$ with $Q(X)\in\Z[X]$. Since $H$ and $\m_{\alpha}$ are both monic in $X$ then by Gauss's lemma $Q(X)$ is also monic. Let $m= \deg \m_\alpha$, and $n=\deg H$. Write 
    \[
Q(X)=X^{n-m}+q_{n-m-1}X^{n-m-1}+\dots+q_0,\ q_0,\dots,q_{n-m}\in \Z.
    \]

We will first prove that all the coefficients $q_i$ are nonnegative by comparing the coefficients of $H(X)$ with $m_\alpha(X).Q(X)$ in an increasing degree. The degree zero coefficient, implies that $-cq_0\leq 0$, which implies that $q_0\geq 0$, the degree $1$ implies that $-cq_1+p_1q_0\leq 0$, which implies that $q_1\geq \frac{p_1q_0}{c}\geq 0$, continuing in the same manner, we find that $\sum\limits_{i+j=n-m-1}p_iq_j-cq_{n-m-1}\leq 0$ which implies that $q_{n-m-1}\geq \frac{\sum\limits_{i+j=n-m-1}p_iq_j}{c}\geq 0$. Which proves the claim. 

Now starting the comparison at degree $n-1$, we have $p_{m-1}+q_{n-m-1}\leq 0$, and since both terms on the left-hand side are non-negative, they must be both zero. Taking the degree $n-2$, we get  $p_{m-2}+q_{n-m-2}\leq 0$, which again implies that both terms must be zero. Continuing in the same manner, if $n-m\geq m$, we get at degree $n-m+1$, $p_1+q_{n-2m+1}\leq 0$, which again implies that both terms must be zero, and this will contradict the fact that not all the $p_i$s are zero. Otherwise, if $n-m<m$, then at degree $m$, we get $p_{m-(n-m)}+q_0\leq 0$, which implies that all the $q_i$s are zero, which is again impossible since this will imply that $\m_\alpha=X^m-c$. 
    
        \end{itemize} 
\end{enumerate}
\end{proof}

\begin{theorem}\label{perron}
    If $\Na$ is a finitely generated monoid then $\alpha$ is a weak Perron number. 
    Conversely, suppose that $\alpha$ is an algebraic integer, that $\alpha$ is weak Perron, and that $\alpha$ is the unique positive conjugate of its minimal polynomial. Then $\Na$ is finitely generated.

\end{theorem}
\begin{proof}
    We begin with the forward direction. Since $\Na$ is an FGM, then there exists $n \in \N$ such that $\alpha^n = \sum\limits_{i=0}^{n-1}c_i \alpha^i$ where $c_i \in \N_0 $. Define $Q(t) = t^n - \sum\limits_{i=0}^{n-1}c_i t^i$. 
    Suppose $\alpha$ is not a weak Perron number, then there exists a conjugate $\beta$ of $\alpha$ such that $|\beta| > \alpha$. Since $Q(\alpha)=0$ we must also have $Q(\beta)=0$. Therefore, $\beta^n=\sum\limits_{i=0}^{n-1}c_i \beta^i$ and by the triangle inequality we get $|\beta|^n \leq \sum\limits_{i=0}^{n-1}c_i |\beta|^i$. $Q(X)$ is ultimately positive and it has only one sign change, thus, by Theorem~\ref{descarte}, it has at most one positive root, which must be $\alpha$. Since $|\beta| > \alpha$ then $Q(|\beta|)>0$ and therefore we get $|\beta|^n > \sum\limits_{i=0}^{n-1}c_i |\beta|^i$ which is a contradiction. \\

    For the reverse direction, assume that $\alpha$ is an algebraic integer, that $\alpha$ is weak Perron, and that $\alpha$ is the unique positive conjugate of its minimal polynomial. For the case $\alpha=1$, $\Na = \N_0$ which is finitely generated. We assume $\alpha \neq 1$.

    We show that $\alpha>1$. Suppose, by contradiction, that $0<\alpha<1$. Let $\alpha=\alpha_1,\alpha_2,\dots,\alpha_d$ be the conjugates of $\alpha$. Since $\alpha$ is weak Perron, we have $|\alpha_i|\leq \alpha<1$ for every $i\in\{1,\dots,d\}$. Hence $|m_\alpha(0)|=\prod_{i=1}^d |\alpha_i| <1$ which is a contradiction since $\alpha$ is an algebraic integer. It follows that $\alpha>1$.
    
    Let $k$ be the number of roots of $m_\alpha(X)$ on the circle $|z|=\alpha$. Since $\alpha$ itself is such a root, we have $k\geq 1$. By Theorem~\ref{ferguson}, there exists a polynomial $g(X)\in \mathbb Z[X]$ such that $m_\alpha(X)=g(X^k)$. Moreover, since $m_\alpha(X)$ is irreducible over $\mathbb Q$, the polynomial $g$ is also irreducible over $\mathbb Q$. Set $\lambda=\alpha^k$. Then $g(\lambda)=0$, and since $\alpha$ is an algebraic integer, so is $\lambda$.

    We claim that $\lambda$ is a Perron number. Let $\delta$ be a conjugate of $\lambda$, equivalently a root of $g$. If $\eta^k=\delta$, then $m_\alpha(\eta)=g(\eta^k)=g(\delta)=0$. Thus $\eta$ is a conjugate of $\alpha$. Since $\alpha$ is weak Perron, we have $|\eta|\leq \alpha$. Therefore $|\delta|=|\eta|^k\leq \alpha^k=\lambda$. Hence $\lambda$ is weak Perron.

    It remains to show that no conjugate of $\lambda$ distinct from $\lambda$ has modulus $\lambda$. Suppose that $\delta$ is a root of $g$ with $|\delta|=\lambda$. Let $\omega_0=1,\omega_1,\dots,\omega_{k-1}$ be the $k$-th distinct roots of unity. We have $m_\alpha(\omega_i\alpha)=g(\alpha^k)=g(\lambda)=0$. Therefore, the $k$-th roots of $m_\alpha$ that lie on the circle $|z|=\alpha$ are $\alpha \omega_0, \alpha \omega_1 \dots,\alpha \omega_{k-1}$. Let $\mu \in \C$ with $\mu ^ k = \delta$. We have $g(\delta)=0$ implies that $g(\mu^k)=0$ implies that $m_\alpha(\mu)=0$ with $|\mu|=\sqrt[k]{|\delta|} = \sqrt[k]{\lambda}=\alpha$. Therefore $\mu = \omega_i \alpha$ for a certain $i$. This implies $\delta=\mu^k=\lambda$. So $\lambda$ is a Perron number.

    We now show that $\lambda$ is the unique positive conjugate of its minimal polynomial. Let $\delta>0$ be a positive root of $g$. Then $\delta^{1/k}>0$ is a positive root of $f$, because $f(\delta^{1/k})=g(\delta)=0$. Since $\alpha$ is the unique positive root of $f$, we get $\delta^{1/k}=\alpha$, and therefore $\delta=\alpha^k=\lambda$. Hence $\lambda$ has no positive conjugate aside from itself.

    If $k \neq 1$, set $\gamma=\lambda^{-1}$. Then $\gamma\in(0,1)$, and $\gamma^{-1}=\lambda$ is a Perron number with no positive conjugate aside from itself. By Proposition~\ref{simple-polynomial}, there exists a polynomial $h(X)\in\mathbb Z[X]$ such that $h(X)m_\gamma(X)\in X\mathbb N_0[X]-1$. Therefore there exists a polynomial $F(X)\in X\mathbb N_0[X]$ such that $h(X)m_\gamma(X)=F(X)-1$. Evaluating at $\gamma$ gives $F(\gamma)=1$. Write $F(X)=\sum_{i=1}^{N}c_iX^i$ where $c_1,\dots,c_N\in\mathbb N_0$. Then $1=\sum_{i=1}^{N}c_i\gamma^i =\sum_{i=1}^{N}c_i\lambda^{-i}$. Multiplying by $\lambda^N$ gives $\lambda^N=\sum_{i=1}^{N}c_i\lambda^{N-i}$. Which implies that $\mathbb N_0[\lambda]$ is finitely generated by
Proposition~\ref{generation}.
    
    By Proposition~\ref{x_k_atoms}, applied to $\lambda$ and to the root $\alpha$ of $g(X^k)$, gives $\left|\mathcal A(\mathbb N_0[\alpha])\right| = k\left|\mathcal A(\mathbb N_0[\lambda])\right|$. Thus $\mathcal A(\mathbb N_0[\alpha])$ is finite and nonempty. By Proposition~\ref{generation}, $\mathbb N_0[\alpha]$ is therefore atomic and finitely generated.

    If $k=1$, set $\gamma=\alpha^{-1}$ and the above argument follows.
\end{proof}

\begin{remark}
The hypotheses in the previous reverse direction cannot be weakened in an obvious way. The condition that $\alpha$ be an algebraic integer cannot be omitted. Indeed, take $\alpha=\frac{3}{2}$. The condition that $\alpha$ be the unique positive conjugate cannot be omitted as well. Take $\alpha=2+\sqrt{2}$.
\end{remark}

\section{Application to Degree 3 Polynomials}\label{deg3}

As before, let $\alpha$ be an algebraic positive real number over $\Q$. Let $\m_{\alpha}$ be the primitive polynomials of minimal degree over $\Z$ with $\m_{\alpha}(\alpha)=0$. In this section, we will restrict our attention to the case $\deg  \m_\alpha = 3$. Let $p_{\alpha}(X),\ q_\alpha(X)\in\N_0[X]$ such that $m_{\alpha}(X)=p_{\alpha}(X)-q_{\alpha}(X)$.

In the case when $\m_\alpha$ is not monic or when $\alpha$ is not the unique positive root, by Proposition \ref{onepositiveroot} $\Na$ is infinitely generated or antimatter. For the rest of this section, we suppose that $\m_\alpha$ is monic with $\alpha$ its unique positive root. 

Also note that if $\alpha < 1$ then for any $n \in \N$, $\alpha^n$ cannot be written as an $\N_0$-linear combination of $\{1, \alpha, \dots, \alpha ^ {n-1}\}$. In view of Theorem \ref{generation}, $\Na$ must be infinitely generated or antimatter. For the rest of this section, we will assume that $\alpha > 1$.

Write $\m_\alpha(X) = X^3 \pm aX^2 \pm bX \pm c$ where $a,b,c \in \N_0$. Note that $c \neq 0$ since $\m_\alpha$ is irreducible. Since $\m_\alpha(X)$ has a unique positive root and it is eventually positive, then $\m_\alpha(0) < 0$ which forces $\m_\alpha$ to have the form $\m_\alpha(x) = X^3 \pm aX^2 \pm bX - c$. We assume that $a,b,c \neq 0$ and we treat the cases where they are zero separately. At the end of this section, there is a summary of the classification.

\subsection{The form \texorpdfstring{$\m_\alpha(X) = X^3+aX^2-bX-c$}{mα(X) = X³+aX²-bX-c}.}

For $\m_\alpha(X) = X^3+aX^2-bX-c$, if $\alpha$ is not a weak Perron number, then $\Na$ is infinitely generated or antimatter by Theorem \ref{perron}. If $\Na$ is an FGM, we have the following necessary conditions on the coefficients of $\m_\alpha(X)$:
\begin{proposition}\label{coef}
    Suppose $\Na$ is an FGM and $\m_\alpha(X) = X^3+aX^2-bX-c$ then $b\geq a^2$ and $b^3 \geq a^3c$    
\end{proposition}
\begin{proof}
    Let $\beta$ and $\gamma$ be the conjugates of $\alpha$ over $\Q$. Since $\alpha$ is a weak Perron number, we have $|\beta| \leq \alpha$ and $|\gamma| \leq \alpha$. \\
    
    By Viète's formulas we have $\alpha\beta\gamma=c$ which implies $c\leq \alpha^3$. We also have $\alpha^3+a\alpha^2-b\alpha = c \leq \alpha^3$ which implies $\frac{b}{a} \geq \alpha$. Combining both inequalities we get $b^3 \geq a^3c$. For the other inequality, we have by Viète's equations $\beta + \gamma + \alpha = -a$ which implies $a + \alpha \leq -\beta - \gamma \leq 2\alpha$. Combining with $\alpha \leq \frac{b}{a}$, the inequality follows. \\
\end{proof}

The converse of the above proposition is not true in general, even if $\Na$ is atomic. In the next example, we will prove that the above conditions do not imply weak perron.

\begin{example}
    Let $\m_\alpha(X)=X^3+X^2-bX-2$ where $b\geq 6$ is an integer. By checking the divisors of $2$, we can see that $\m_\alpha(X)$ has no rational root, and hence is irreducible. We have $\m_\alpha(\sqrt{b}).\m_\alpha(0)<0$, thus $\alpha$ is between $0$ and $\sqrt{b}$. Note that $\alpha$ is the unique positive root for $\m_\alpha$ since by Theorem \ref{descarte} it has $1$.
    We also have $\m_\alpha(-\sqrt{b})>0$ with $\m_\alpha$ of odd degree, so it must have a root $\beta<-\sqrt{b}$. Thus we have $|\beta|>\sqrt{b}>\alpha$, and therefore $\alpha$ is not a weak Perron number. Since $|\m_\alpha(0)|\neq 1$, then $\Na$ is atomic by proposition \ref{atomicity}. In view of theorem \ref{perron}, $\Na$ is infinitely generated.
\end{example}

\begin{proposition}\label{notlfm}
    If $\Na$ is atomic and $\m_\alpha(X) = X^3+aX^2-bX-c$ then $\Na$ is not an LFM.
\end{proposition}
\begin{proof}
    By Proposition \ref{ufm}, $\Na$ is not a UFM. Suppose it is a proper LFM, by Proposition \ref{lfm} with Proposition \ref{generation} there exists a polynomial $H(X) = X^4 - \sum \limits_{i=0}^{3} c_i X^i$ where $c_i$ are all nonnegative having $\alpha$ as a root. As such, it can be written as $H(X) = \m_\alpha(X)(X+d)$ with $d$ being a nonzero integer. After developing, considering the coefficient degree $0$, we get $d > 0$. The coefficient of degree $3$ implies that $d < -a$ which is a contradiction. 
\end{proof}

The next example illustrates the case when $\m_\alpha$ satisfies the necessary conditions of proposition \ref{coef}, and the size of the set of atoms is $5$.

\begin{example}
    Let $\m_\alpha(X)=X^3+X^2-pX-2p$ where $p$ is a prime number strictly greater than $3$. By the rational root test, the only possible rational roots are $\pm1,\pm2,\pm p,\pm2p$, and none of them is a root, and hence is irreducible. We have $\m_\alpha(1)<0$, and $\m_\alpha$ is ultimately positive, therefore it has a positive root $\alpha>1$. Note that $\alpha$ is the unique positive root for $\m_\alpha$ since by Theorem \ref{descarte} it can have at most $1$. Since $|\m_\alpha(0)|\neq 1$, then $\Na$ is atomic by Proposition~\ref{atomicity}. In view of Proposition~\ref{notlfm}, $\Na$ is not an $LFM$, thus $1,\ \alpha,\ \alpha^2,\ \alpha^3,\ \alpha^4$ are atoms. We will prove that $ \alpha^5$ is not an atom. Indeed, $\alpha$ is a root of $X^5+(1-p)X^3+(2-p)X^2-4p$. 
\end{example}

\subsection{The form \texorpdfstring{$\m_\alpha(X) = X^3-aX^2+bX-c$}{mα(X) = X³-aX²+bX-c}.}

Suppose $\m_\alpha(X) = X^3-aX^2+bX-c$, with a positive root $\alpha$ such that $\Na$ is atomic.
Since $\m_\alpha(-x)$ has no sign changes, by theorem \ref{descarte}, $\m_\alpha$ has no negative roots. If it has three positive roots, then by Proposition~\ref{onepositiveroot} it cannot be finitely generated. For the rest of this subsection, we suppose that $\alpha$ is the unique positive root of $\m_\alpha$. Denote $\beta$ and $\overline{\beta}$ the conjugates of $\alpha$.\\

We have the following lemma:

\begin{lemma}\label{coefpmpm}
    The root $\alpha$ is a weak perron if and only if $b^3\leq a^3c$.    
\end{lemma}
\begin{proof}
    By Viète's equations, we have $c=\alpha\beta\overline{\beta}$. We deduce the following: \begin{align*}
        \alpha\ \text{is a weak perron} &\Leftrightarrow c\leq \alpha^3\\
        &\Leftrightarrow \sqrt[3]{c}\leq \alpha\\
        &\Leftrightarrow \m_\alpha(\sqrt[3]{c})\leq 0\ \text{by continuity of }\m_\alpha\\
        &\Leftrightarrow b^3\leq a^3c.
    \end{align*}
\end{proof}

\begin{corollary}\label{infgen-cubic}
Assume that $\m_\alpha(X)=X^3-aX^2+bX-c$ and that $\alpha$ is the unique positive root of $\m_\alpha$. Then $\Na$ is infinitely generated if and only if $b^3>a^3c$.

\end{corollary}

\begin{proof}
First suppose that $b^3>a^3c$. Then $\alpha$ is not weak Perron, by Lemma~\ref{coefpmpm}. By Theorem~\ref{perron}, $\Na$ is infinitely generated.

Conversely, suppose that $b^3\leq a^3c$. Then Lemma~\ref{coefpmpm} implies that $\alpha$ is weak Perron. Since $\alpha$ is an algebraic integer and is the unique positive conjugate of its minimal polynomial, then the converse implication of Theorem~\ref{perron}
applies, and so $\Na$ is finitely generated. 
\end{proof}

We deduce a family of infinitely generated monoids $\N_0[\alpha_p]$ of rank $3$, where $\alpha_p$ is a positive root of $X^3-X^2+pX-2p^2$, where $p>2$ is a prime. 

\begin{proposition}\label{properlfm}
    $\Na$ is a proper LFM if and only if $b\leq a^2$ and $\frac{b}{a}\leq \lfloor\frac{c}{b}\rfloor$.
\end{proposition}
\begin{proof}
    First note that $\Na$ is not a $UFM$ by proposition \ref{ufm}. Let $Q(X)=X+d$ where $d\in\Z$. For $\m_\alpha Q$ to be a negative tail polynomial it is necessary and sufficient that the following system admits an integer solution for $d$:
    \begin{align*}
        -cd\leq 0,\ bd-c\leq 0,\ -ad+b\leq 0,\ -a+d\leq 0,
    \end{align*}
    this system is solvable if and only if $b\leq a^2$ and $\frac{b}{a}\leq \lfloor\frac{c}{b}\rfloor$.
\end{proof}

We note that in the case when $b>a^2$, one can prove that $\left|\mathcal{A}(\Na)\right|\geq 6$. The next two examples illustrate cases where $\left|\mathcal{A}(\Na)\right|= 5$ and $\left|\mathcal{A}(\Na)\right|= 7$.

\begin{example}
    In this example $\alpha$ is a root of $\m_\alpha(X)=X^3-3X^2+5X-8$ and $\left|\mathcal{A}(\Na)\right|= 5$. 
\end{example}

\begin{example}
    In this example $\alpha$ is a root of $\m_\alpha(X)=X^3-2X^2+5X-20$ and $\left|\mathcal{A}(\Na)\right|= 7$.
\end{example}

\begin{table}[ht]
\centering
\caption{Cubic sign patterns and the corresponding generation behavior of $\Na$.}
\label{tab:cubic-classification}
\renewcommand{\arraystretch}{1.25}
\begin{tabular}{@{}
>{\centering\arraybackslash}p{0.26\textwidth}
>{\raggedright\arraybackslash}p{0.5\textwidth}
>{\raggedright\arraybackslash}p{0.20\textwidth}
@{}}
\toprule
Form of $\m_\alpha(X)$ & Conclusion for $\Na$ & Reference \\
\midrule

$X^3-c$ 
& 
UFM whenever $X^3-c$ is irreducible. 
& 
Proposition~\ref{ufm}

\\

$X^3+bX-c$
&
Antimatter if $c=1$; infinitely generated if $c>1$.
&
Theorem~\ref{pminusc}
\\

$X^3-bX-c$
&
UFM.
&
Proposition~\ref{ufm}
\\

$X^3+aX^2-c$
&
Antimatter if $c=1$; infinitely generated if $c>1$.
&
Theorem~\ref{pminusc}
\\

$X^3-aX^2-c$
&
UFM.
&
Proposition~\ref{ufm}
\\

$X^3+aX^2+bX-c$
&
Antimatter if $c=1$; infinitely generated if $c>1$.
&
Theorem~\ref{pminusc}
\\

$X^3-aX^2-bX-c$
&
UFM.
&
Proposition~\ref{ufm}
\\

\midrule

\multirow{2}{*}{$X^3+aX^2-bX-c$}
&
If $\Na$ is an FGM, then $b\ge a^2$ and $b^3\ge a^3c$.
&
Proposition~\ref{coef}
\\
\cmidrule(l){2-3}
&
If $\Na$ is atomic, then $\Na$ is never an LFM.
&
Proposition~\ref{notlfm}
\\

\midrule

\multirow{2}{*}{$X^3-aX^2+bX-c$}
&
$\Na$ is a proper LFM if and only if $b\le a^2$ and $\frac{b}{a}\le \left\lfloor \frac{c}{b}\right\rfloor$.

&
Proposition~\ref{properlfm}
\\
\cmidrule(l){2-3}
&
$b^3 \leq a^3c$ if and only if $\Na$ is finitely generated.
&
Corollary~\ref{infgen-cubic}
\\

\bottomrule
\end{tabular}
\end{table}
\newpage
\section*{Acknowledgements}

The authors are grateful to the anonymous referee for a careful reading of the manuscript and for several valuable comments. In particular, the referee encouraged us to reconsider the converse direction relating weak Perron numbers to finite generation.


\begin{thebibliography}{1}

\bibitem{ajran2023factorization}
K.~Ajran, J.~Bringas, B.~Li, E.~Singer, and M.~Tirador,
\newblock ``Factorization in additive monoids of evaluation polynomial semirings,''
\newblock {\em Communications in Algebra}, vol. 51, no. 10, pp. 4347--4362, 2023, Taylor \& Francis.

\bibitem{Bourne1951}
S.~Bourne,
\newblock \emph{The Jacobson radical of a semiring},
\newblock Proc. Natl. Acad. Sci. USA \textbf{37} (1951), no.~3, 163--170.
\newblock \href{https://doi.org/10.1073/pnas.37.3.163}{doi:10.1073/pnas.37.3.163}.

\bibitem{chapman2026betti}
S.~T.~Chapman, P.~García-Sánchez, C.~O'Neill, and V.~Ponomarenko,
\newblock ``Betti elements and full atomic support in rings and monoids,''
\newblock {\em Journal of Algebra}, Elsevier, 2026.

\bibitem{CMG}
J.~Correa-Morris and F.~Gotti,
\newblock \emph{On the additive structure of algebraic valuations of polynomial semirings},
\newblock Journal of Pure and Applied Algebra \textbf{226} (2022), no.~11, 107104.
\newblock \href{https://doi.org/10.1016/j.jpaa.2022.107104}{doi:10.1016/j.jpaa.2022.107104}.

\bibitem{curtiss1918recent}
D.~R.~Curtiss,
\textit{Recent Extensions of Descartes' Rule of Signs},
Annals of Mathematics,
\textbf{19} (1918), no. 4, 251--278.

\bibitem{dani2025set}
J.~Dani, A.~Deng, M.~Gotti, B.~Li, A.~Paladiya, J.~Vulakh, and J.~Zeng,
\newblock ``On the set of atoms and strong atoms in additive monoids of cyclic semidomains,''
\newblock {\em arXiv preprint arXiv:2508.11319}, 2025.

\bibitem{GHK2006}
A.~Geroldinger and F.~Halter-Koch,
\newblock \emph{Non-Unique Factorizations: Algebraic, Combinatorial and Analytic Theory},
\newblock Chapman \& Hall/CRC, 2006.
\newblock \href{https://doi.org/10.1201/9781420003208}{doi:10.1201/9781420003208}.

\bibitem{Golan1999}
J.~S.~Golan,
\newblock \emph{Semirings and their Applications},
\newblock Kluwer Academic Publishers, 1999.
\newblock \href{https://doi.org/10.1007/978-94-015-9333-5}{doi:10.1007/978-94-015-9333-5}.

\bibitem{gotti2023subatomicity}
F.~Gotti and H.~Polo,
\newblock ``On the subatomicity of polynomial semidomains,''
\newblock in {\em Algebraic, Number Theoretic, and Topological Aspects of Ring Theory}, pp. 197--212, Springer, 2023.

\bibitem{HebischWeinert1998}
U.~Hebisch and H.~J.~Weinert,
\newblock \emph{Semirings: Algebraic Theory and Applications in Computer Science},
\newblock World Scientific, 1998.
\newblock \href{https://doi.org/10.1142/3903}{doi:10.1142/3903}.

\bibitem{jiang2023primality}
N.~Jiang, B.~Li, and S.~Zhu,
\newblock ``On the primality and elasticity of algebraic valuations of cyclic free semirings,''
\newblock {\em International Journal of Algebra and Computation}, vol. 33, no. 2, pp. 197--210, 2023, World Scientific. 

\bibitem{Lind1984}
D.~A.~Lind,
\newblock \emph{The entropies of topological Markov shifts and a related class of algebraic integers},
\newblock Ergodic Theory Dynam. Systems \textbf{4} (1984), no.~2, 283--300.
\newblock \href{https://doi.org/10.1017/S0143385700002443}{doi:10.1017/S0143385700002443}.

\bibitem{RosalesGarciaSanchez2009}
J.~C.~Rosales and P.~A.~Garc\'ia-S\'anchez,
\newblock \emph{Numerical Semigroups},
\newblock Developments in Mathematics, vol.~20, Springer, 2009.
\newblock \href{https://doi.org/10.1007/978-1-4419-0160-6}{doi:10.1007/978-1-4419-0160-6}.

\bibitem{Vandiver1934}
H.~S.~Vandiver,
\newblock \emph{Note on a simple type of algebra in which the cancellation law of addition does not hold},
\newblock Bull. Amer. Math. Soc. \textbf{40} (1934), no.~12, 914--920.
\newblock \href{https://doi.org/10.1090/S0002-9904-1934-06003-8}{doi:10.1090/S0002-9904-1934-06003-8}.

\bibitem{Ferguson1997}
R. Ferguson,
``Irreducible polynomials with many roots of equal modulus,''
\textit{Acta Arithmetica}, \textbf{78} (1997), no.~3, 221--225.
\bibitem{ChenGottiLuYao2026}
T. Chen, F. Gotti, T. Lu, and A. Yao,
\textit{On the Additive Structure of Simple Semiring Extensions},
PRIMES Research Paper, Massachusetts Institute of Technology, 2026.

\end{thebibliography}
\end{document}